\documentclass[12pt,a4paper]{amsart}

\newtheorem{theo+}              {Theorem}           [section]
\newtheorem{prop+}  [theo+]     {Proposition}
\newtheorem{coro+}  [theo+]     {Corollary}
\newtheorem{lemm+}  [theo+]     {Lemma}
\newtheorem{exam+}  [theo+]     {Example}
\newtheorem{rema+}  [theo+]     {Remark}
\newtheorem{defi+}  [theo+]     {Definition}
\def \r{\mbox{${\mathbb R}$}}

\newenvironment{theorem}{\begin{theo+}}{\end{theo+}}
\newenvironment{proposition}{\begin{prop+}}{\end{prop+}}
\newenvironment{corollary}{\begin{coro+}}{\end{coro+}}
\newenvironment{lemma}{\begin{lemm+}}{\end{lemm+}}
\newenvironment{definition}{\begin{defi+}}{\end{defi+}}

\usepackage{amsthm}
\theoremstyle{plain} \theoremstyle{remark}
\newtheorem{remark}{Remark}
\newtheorem{example}{Example}

\def\E{/\kern-1.0em \equiv }

\title{On $f$-biharmonic maps and $f$-biharmonic submanifolds}
\author{Ye-Lin Ou $^{*}$}
\thanks{$^{*}$ Research supported by
NSF of Guangxi (P. R. China), 2011GXNSFA018127.}
\address{Department of
Mathematics,\newline\indent Texas A $\&$ M University-Commerce,
\newline\indent Commerce, TX 75429,\newline\indent USA.\newline\indent
E-mail:yelin.ou@tamuc.edu }
\begin{document}

\title[On $f$-biharmonic maps and $f$-biharmonic submanifolds]{On $f$-biharmonic maps and $f$-biharmonic submanifolds}

\date{06/12/2013}
\maketitle

\section*{Abstract}
\begin{quote}
{\footnotesize $f$-Biharmonic maps are the extrema of the $f$-bienergy functional. $f$-biharmonic submanifolds are submanifolds whose defining isometric immersions are $f$-biharmonic maps. In this paper, we prove that an $f$-biharmonic map from a compact Riemannian manifold into a non-positively curved manifold with constant $f$-bienergy density is a harmonic map; any $f$-biharmonic function on a compact manifold is constant, and that the inversions about $S^m$ for $m\ge 3$ are proper $f$-biharmonic conformal diffeomorphisms. We derive $f$-biharmonic submanifolds equations and prove that a surface in a manifold $(N^n, h)$ is an $f$-biharmonic surface if and only it can be biharmonically conformally immersed into $(N^n,h)$. We also give a complete classification of $f$-biharmonic curves in $3$-dimensional Euclidean space. Many examples of proper $f$-biharmonic maps and $f$-biharmonic surfaces and curves are given.} 
\end{quote}
\section{Harmonic, biharmonic, $f$-harmonic and $f$-biharmonic maps}
All objects including manifolds, tension fields, and maps in this paper are assumed to be smooth if no otherwise statement is given.\\
{\bf Harmonic maps and their equations:} Harmonic maps are critical points of the energy functional for maps $\phi: (M,g)\longrightarrow (N,h)$ between Riemannian manifolds:
\begin{equation}\notag
E(\phi)=\frac{1}{2}\int_\Omega |{\rm d}\phi|^2v_g,
\end{equation}
where $\Omega$ is a compact domain of $ M$. The Euler-Lagrange
equation gives the harmonic map equation ( \cite{ES})
\begin{equation}\label{fhm}
\tau(\phi) \equiv {\rm Tr}_g\nabla\,d \phi=0,
\end{equation}
where $\tau(\phi)={\rm Tr}_g\nabla\,d \phi$ is called the tension field of
the map $\phi$.\\
{\bf Biharmonic maps and their equations:} Biharmonic maps are critical points of the bienergy functional for maps $\phi: (M,g)\longrightarrow (N,h)$ between Riemannian manifolds:
\begin{equation}\notag
E_{2}(\phi)=\frac{1}{2}\int_\Omega |\tau(\phi)|^2v_g,
\end{equation}
where $\Omega$ is a compact domain of $ M$. The
Euler-Lagrange equation of this functional gives the biharmonic map
equation (\cite{Ji1})
\begin{equation}\label{BTF}
\tau_{2}(\phi):={\rm
Trace}_{g}(\nabla^{\phi}\nabla^{\phi}-\nabla^{\phi}_{\nabla^{M}})\tau(\phi)
- {\rm Trace}_{g} R^{N}({\rm d}\phi, \tau(\phi)){\rm d}\phi =0,
\end{equation}
where $R^{N}$ denotes the curvature operator of $(N, h)$ defined by
$$R^{N}(X,Y)Z=
[\nabla^{N}_{X},\nabla^{N}_{Y}]Z-\nabla^{N}_{[X,Y]}Z.$$\\
{\bf f-Harmonic maps and their equations:} $f$-Harmonic maps are critical points of the $f$-energy functional for maps $\phi: (M,g)\longrightarrow (N,h)$ between Riemannian manifolds:
\begin{equation}\notag
E_{f}(\phi)=\frac{1}{2}\int_\Omega f\,|{\rm d}\phi|^2v_g,
\end{equation}
where $\Omega$ is a compact domain of $ M$. The Euler-Lagrange
equation gives the $f$-harmonic map equation ( \cite{Co},
\cite{OND})
\begin{equation}\label{fhm}
\tau_f(\phi)\equiv f\tau(\phi)+{\rm d}\phi({\rm grad}\,f)=0,
\end{equation}
where $\tau(\phi)={\rm Tr}_g\nabla\,d \phi$ is the tension field of
$\phi$ vanishing of which means $\phi$ is a harmonic map.\\
{\bf f-Biharmonic maps and their equations:} $f$-Biharmonic maps are critical points of the $f$-bienergy functional for maps $\phi: (M,g)\longrightarrow (N,h)$ between Riemannian manifolds:
\begin{equation}\notag
E_{2,f}(\phi)=\frac{1}{2}\int_\Omega f\,|\tau(\phi)|^2v_g,
\end{equation}
where $\Omega$ is a compact domain of $ M$. The Euler-Lagrange
equation gives the $f$-biharmonic map equation ( \cite{Lu})
\begin{equation}\label{FBH}
\tau_{2,f} (\phi)\equiv f\tau_2(\phi)+(\Delta f)\tau(\phi)+2\nabla^{\phi}_{{\rm grad}\, f}\tau(\phi)=0,
\end{equation}
where $\tau(\phi)$ and $\tau_2( \phi)$ are the tension and bitension fields of
$\phi$ respectively.\\
{\bf Bi-f-harmonic maps and their equations:} Bi-$f$-harmonic maps are critical points of the bi-$f$-energy functional for maps $\phi: (M,g)\longrightarrow (N,h)$ between Riemannian manifolds:
\begin{equation}\label{Bif}
E^2_{f}(\phi)=\frac{1}{2}\int_\Omega |\tau_f (\phi)|^2v_g,
\end{equation}
where $\Omega$ is a compact domain of $ M$. The Euler-Lagrange
equation gives the bi-$f$-harmonic map equation ( \cite{OND})
\begin{equation}\label{fbih}
\tau^2_{f} (\phi)\equiv f J^{\phi}(\tau_f(\phi))-\nabla^{\phi}_{{\rm grad}\, f}\tau_f(\phi)=0,
\end{equation}
where $\tau_f(\phi)$ is the f-tension field of the map
$\phi$ and $J^{\phi}$ is the Jacobi operator of the map defined by $J^{\phi}(X)=-[{\rm Tr}_g\nabla^{\phi}\nabla^{\phi}X-\nabla^{\phi}_{\nabla^M}X-R^N(d \phi,\;X)d \phi$].
\begin{remark}
Note that in \cite{OND}, the authors used the name ``$f$-biharmonic maps" for the critical points of the functional (\ref{Bif}). We think that it is more reasonable to call them ``bi-$f$-harmonic maps" as parallel to ``biharmonic maps".
\end{remark}
Clearly, we have the following relationships among these different types of harmonic maps:
$$\{{\rm \bf Harmonic\; maps}\}\subset \{{\rm \bf Biharmonic \;maps} \}\subset \{{\rm \bf f-Biharmonic \;maps}\},$$
$$\{{\rm \bf Harmonic\; maps}\}\subset \{{\rm \bf f-harmonic \;maps}\} \subset \{{\rm \bf Bi-f-harmonic \;maps}\}.$$

From now on we will call an $f$-biharmonic map which is neither harmonic nor biharmonic a {\em proper $f$-biharmonic map}.\\

Harmonic maps as a generalization of important concepts of geodesics, minimal surfaces, and harmonic functions have been studied extensively with tremendous progress in the past 40 plus years. There are volumes of books and literatures about the beautiful theory, important applications, and interesting links of harmonic maps to other areas of mathematics and theoretical physics including nonlinear partial differential equations, holomorphic maps in several complex variables, the theory of stochastic processes, liquid crystals in materials science, and the nonlinear field theory. \\

The study of biharmonic maps was proposed in \cite{EL} and Jiang \cite{Ji1}, \cite{Ji2}, \cite{Ji3} made a first serious study on these maps by using the first and second variational formulas of the bienergy functional and specializing on the biharmonic isometric immersions which nowadays are called biharmonic submanifolds. Very interestingly, the concept of biharmonic submanifolds was also introduced in a different way by B. Y. Chen in \cite{Ch} in his program of understanding the finite type submanifolds in Euclidean spaces. Since 2000, biharmonic maps have been receiving a growing attention and have become a popular subject of study with many progresses. For some recent geometric study of general biharmonic maps see \cite{BK}, \cite{BFO2}, \cite{BMO2}, \cite{MO}, \cite{NUG}, \cite{Ou1}, \cite{Ou4}, \cite{OL}, \cite{Oua} and the references therein. For some recent study of biharmonic submanifolds see \cite{CI}, \cite{Ji2}, \cite{Ji3}, \cite{Di}, \cite{CMO1}, \cite{CMO2}, \cite{BMO1}, \cite{BMO3}, \cite{Ou3}, \cite{OT}, \cite{OW}, \cite{NU}, \cite{TO}, \cite{CM}, \cite{AGR} and the references therein. For biharmonic conformal immersions and submersions see \cite{Ou2}, \cite{Ou5}, \cite{BFO1}, \cite{LO}, \cite{WO} and the references therein.\\

$f$-Biharmonic maps was introduced in \cite{Lu} where the author calculated the first variation to obtain the $f$-biharmonic map equation and the equation for the $f$-biharmonic conformal maps between the same dimensional manifolds. In this paper, we study some basic properties of $f$-biharmonic maps and introduce the concept of $f$-biharmonic submanifolds. We prove that an $f$-biharmonic map from a compact Riemannian manifold into a non-positively curved manifold with constant $f$-bienergy density is a harmonic map (Theorem \ref{NPC}); any $f$-biharmonic function on a compact manifold is constant (Corollary \ref{C25}), and that the inversions about $S^m$ for $m\ge 3$ are proper $f$-biharmonic conformal diffeomorphisms (Proposition \ref{INVERSE}). We derive $f$-biharmonic submanifolds equations (Theorem \ref{HSF} and Corollary \ref{CSMFLD}) and prove that a surface in a manifold $(N^n, h)$ is an $f$-biharmonic surface if and only it can be biharmonically conformally immersed into $(N^n, h)$ (Corollary \ref{CMC}). We also give a complete classification of $f$-biharmonic curves in $3$-dimensional Euclidean space (Theorem \ref{R3}) according to which proper $f$-biharmonic curves are some special subclasses of planar curves or general helices in $\r^3$. Many examples of proper $f$-biharmonic maps and $f$-biharmonic surfaces and curves are given.
\section{Some properties and examples of $f$-biharmonic maps}

$f$-Biharmonic maps are critical points of the $f$-bienergy functional for maps $\phi: (M,g)\longrightarrow (N,h)$ between Riemannian manifolds:
\begin{equation}\notag
E_{2,f}(\phi)=\frac{1}{2}\int_\Omega f\,|\tau(\phi)|^2v_g,
\end{equation}
where $\Omega$ is a compact domain of $ M$.\\

The following theorem was proved in \cite{Lu}. Here, for completeness, we give a brief outline of the proof. Note also  that our notations are different from those in \cite{Lu} and it will show that ours are more compactible.\\
\begin{theorem} 
A map $\phi: (M,g)\longrightarrow (N,h)$ between Riemannian manifolds is an $f$-biharmonic map if and only if
\begin{equation}\label{FBH}
\tau_{2,f} (\phi)\equiv f\tau_2(\phi)+(\Delta f)\tau(\phi)+2\nabla^{\phi}_{{\rm grad}\, f}\tau(\phi)=0,
\end{equation}
where $\tau(\phi)$ and $\tau_2( \phi)$ are the tension and bitension fields of
$\phi$ respectively. 
\end{theorem}
\begin{proof}
Noting that the function $f$ is not change (so does not depend on $t$) during a variation of the $f$-bienergy we can use the standard method (see, e.g., \cite{BK} or \cite{Ji1}) of calculating the first variation of the bienergy functional to have
\begin{eqnarray}
\frac{\partial}{\partial t}E_{2,f}(\phi_t)|_{t=0} &=&\frac{1}{2}\int_{\Omega} f \left\{ \frac{\partial}{\partial t} \langle \tau(\phi_t), \tau(\phi_t)\rangle \right\}_{t=0} v_g \\\notag &=& -\int_{\Omega} f \langle \tau(\phi), J^{\phi}(V)\rangle  v_g\\\notag
 &=& \int_{\Omega} \langle f \tau(\phi), {\rm Tr}_g\nabla^{\phi}\nabla^{\phi}V-\nabla^{\phi}_{\nabla^M}V-R^N(d \phi,V)d \phi\rangle  v_g.
\end{eqnarray}
Using the symmetry property of the curvature tensor and the divergence theorem we can switch the positions of $V$ and $f\tau(\phi)$ to have
\begin{eqnarray}
\frac{\partial}{\partial t}E_{2,f}(\phi_t)|_{t=0} &=& -\int_{\Omega} \langle V, J^{\phi}(f \tau(\phi))\rangle  v_g.
\end{eqnarray}
It follows that $\phi$ is an $f$-biharmonic map if and only if the $f$-bitension field vanishes identically, i.e., $\tau_{2,f}(\phi)=- J^{\phi}(f \tau(\phi))\equiv 0$. Finally, Using Formula (7) in \cite{Ou1}, we have
\begin{eqnarray}\notag
\tau_{2,f}(\phi)&=&- J^{\phi}(f \tau(\phi))=-\{fJ^{\phi}( \tau(\phi))-(\Delta f)\tau(\phi)-2\nabla^{\phi}_{\rm grad f}\tau(\phi)\}\\
&=& f\tau_2(\phi)+(\Delta f)\tau(\phi)+2\nabla^{\phi}_{{\rm grad}\, f}\tau(\phi),
\end{eqnarray}
from which the $f$-biharmonic map equation (\ref{FBH}) follows.
\end{proof}

It is well known that for $m\ne 2$, the harmonicity and $f$-harmonicity of a map $\phi: (M^m,g)\longrightarrow (N^n, h)$  are related via a conformal change of the domain metric. More precisely, we have 
\begin{proposition}\cite{Li}\label{C2}
A map $\phi: (M^m,g)\longrightarrow (N^n,h)$ with $m\ne 2$ is
$f$-harmonic if and only if $\phi:
(M^m,f^{\frac{2}{m-2}}g)\longrightarrow (N^n,h)$ is a harmonic map.
\end{proposition}
In general, this does not generalize to the case of the relationship between biharmonicity and $f$-biharmonicity, but very interestingly, we have
\begin{theorem}\label{CONF}
A map $\phi: (M^2, g)\longrightarrow (N^n, h)$ is
an $f$-biharmonic map if and only if $\phi:
(M^2, f^{-1}g)\longrightarrow (N^n,h)$ is a biharmonic map. 
\end{theorem}
\begin{proof}
On the one hand, we notice that the map $\phi: (M^2, g)\longrightarrow (N^n, h)$ is
an $f$-biharmonic map if and only if
\begin{equation}
f \tau_2(\phi, g)+(\Delta f)\tau(\phi, g)+2\nabla^{\phi}_{{\rm grad} f}\tau(\phi, g)=0,
\end{equation}
which is equivalent to
\begin{equation}\label{2D}
\tau_2(\phi, g)+(\Delta \ln f +|\rm grad \ln f|^2))\tau(\phi, g)+2\nabla^{\phi}_{{\rm grad} \ln f}\tau(\phi, g)=0.
\end{equation}
On the other hand, by Corollary 1 in \cite{Ou2}, the relationship between the bitension field $\tau_2(\phi, g)$ and that of the map $\phi: (M^2, \bar g= F^{-2}g)\longrightarrow (N^n, h)$ is given by
\begin{eqnarray}\notag
\tau_{2}(\phi,{\bar g})= F^4\{\tau^{2}(\phi, g) +
2(\Delta {\rm ln}F+2\left|{\rm grad\,ln}F\right|^2)\tau(\phi,
g)+4\nabla^{ \phi}_{{\rm grad\,ln}\,F}\,\tau(\phi,
g)\},
\end{eqnarray}
which is equivalent to
\begin{eqnarray}\notag
\tau_{2}(\phi,{\bar g})= F^4\{\tau^{2}(\phi, g) +
(\Delta {\rm ln}F^2+\left|{\rm grad\,ln}F^2\right|^2)\tau(\phi,
g)+2\nabla^{ \phi}_{{\rm grad\,ln}\,F^2}\,\tau(\phi, g)\}
\end{eqnarray}
It follows that the map $\phi: (M^2, \bar g= F^{-2}g)\longrightarrow (N^n, h)$ is biharmonic if and only if
\begin{eqnarray}\label{CF}
\tau_{2}(\phi, g) +(\Delta {\rm ln}F^2+\left|{\rm grad\,ln}F^2\right|^2)\tau(\phi, g)+2\nabla^{ \phi}_{{\rm grad\,ln}\,F^2}\,\tau(\phi, g)=0.
\end{eqnarray}
Substituting $F^2=f$ into Equation (\ref{CF}) we obtain Equation (\ref{2D}). Thus, we conclude that the map $\phi: (M^2, g)\longrightarrow (N^n, h)$ is
an $f$-biharmonic map if and only if $\phi:
(M^m, f^{-1}g)\longrightarrow (N^n,h)$ is a biharmonic map. This gives the theorem.
\end{proof}
\begin{theorem}\label{NPC}
Any $f$-biharmonic map $\phi: (M^m, g)\longrightarrow (N^n, h)$ from a compact Riemannian manifold into a non-positively curved manifold with constant $f$-bienergy density (i.e., $f|\tau(\phi)|^2=C$) is a harmonic map.
\end{theorem}
\begin{proof}
A straightforward computation gives
\begin{eqnarray}\notag
\Delta\left (\frac{1}{2}f|\tau(\phi)|^2\right)&=&\frac{1}{2}\Delta\langle f^{1/2}\tau(\phi), \;f^{1/2}\tau(\phi)\rangle=(\nabla^{\phi}_{e_1}\nabla^{\phi}_{e_1}-\nabla^{\phi}_{\nabla^M_{e_i}e_i})\langle f^{1/2}\tau(\phi), \;f^{1/2}\tau(\phi)\rangle\\\notag
&=&\langle \nabla^{\phi}_{e_i} f^{1/2}\tau(\phi), \; \nabla^{\phi}_{e_i} f^{1/2}\tau(\phi)\rangle+\langle (\nabla^{\phi}_{e_i}\nabla^{\phi}_{e_i}-\nabla^{\phi}_{\nabla^M_{e_i}e_i} f^{1/2}\tau(\phi), \; f^{1/2}\tau(\phi)\rangle\\\label{GD9}
&=&\langle \nabla^{\phi}_{e_i} f^{1/2}\tau(\phi), \; \nabla^{\phi}_{e_i} f^{1/2}\tau(\phi)\rangle+f\langle (\nabla^{\phi}_{e_i}\nabla^{\phi}_{e_i}-\nabla^{\phi}_{\nabla^M_{e_i}e_i}\tau(\phi), \; \tau(\phi)\rangle+\\\notag
&+&f^{1/2}(\Delta f^{1/2})|\tau(\phi)|^2+2f^{1/2}\langle\nabla^{\phi}_{{\rm grad}f^{1/2}}\tau(\phi), \tau(\phi)\rangle.
\end{eqnarray}
Since $\phi$ is assumed to be $f$-biharmonic we have
\begin{eqnarray}\label{GD10}
&& f\langle (\nabla^{\phi}_{e_i}\nabla^{\phi}_{e_i}-\nabla^{\phi}_{\nabla^M_{e_i}e_i}\tau(\phi), \; \tau(\phi)\rangle\\\notag=&&\langle f R^N(d\phi(e_i), \tau(\phi))d\phi(e_i) -(\Delta f )\tau(\phi)-2\nabla^{\phi}_{{\rm grad}f}\tau(\phi), \tau(\phi)\rangle\\\notag
=&& f\langle R^N(d\phi(e_i), \tau(\phi))d\phi(e_i) , \tau(\phi)\rangle-(\Delta f )|\tau(\phi)|^2-2\langle\nabla^{\phi}_{{\rm grad}f}\tau(\phi), \tau(\phi)\rangle.
\end{eqnarray}
Substituting (\ref{GD10}) into (\ref{GD9}) and simplifying the result gives
\begin{eqnarray}\label{GD18}
&&\Delta\left (\frac{1}{2}f|\tau(\phi)|^2\right)\\\notag && = f|\nabla^{\phi}_{e_i}\tau(\phi)|^2-f R^N(d\phi(e_i), \tau(\phi), d\phi(e_i) , \tau(\phi))-\frac{1}{2}(\Delta f) |\tau(\phi)|^2. 
\end{eqnarray}
Using Equation (\ref{GD18}), together with the assumptions that $f|\tau(\phi)|^2=C$, $f> 0$, and $R^N(d\phi(e_i), \tau(\phi), d\phi(e_i) , \tau(\phi))\le 0$, we conclude that $f$ is a subharmonic function on the compact manifold $(M, g)$ and hence $f$ is a constant function. It follows that the $f$-biharmonic maps $\phi$ is actually a biharmonic map from a compact manifold into a nonpositively curved manifold which has been proved to be a harmonic map by a theorem of Jiang in (\cite{Ji1}). Thus, we obtain the theorem. 
\end{proof}
\begin{remark}
Note that there are many harmonic maps between spheres with constant energy density (called eigenmaps). As our Theorem \ref{NPC} implies that there is no proper $f$-biharmonic maps from a compact manifold into a nonpositively curved manifold with constant $f$-bienergy density, it would be interesting to know if there is any proper $f$-biharmonic map between spheres with constant $f$-bienergy density.
\end{remark}
\begin{proposition}\label{P3}
A function $u:(M,g)\longrightarrow \r$ is $f$-biharmonic if and only if 
\begin{eqnarray}\label{FBHF}
&& f\Delta^2 u+(\Delta f) \Delta u+2g({\rm grad} f, {\rm grad}\Delta u)=0, \;{\rm or,\; equivalently}\\\label{FBHF2}
&&\Delta (f\Delta u)=0,
\end{eqnarray}
where $\Delta^2 u=\Delta(\Delta u)$ denotes the bi-Laplacian of $u$. In other words, a function $u$ is an $f$-biharmonic function if and only if the product $f\Delta u$ is a harmonic function. In particular, a quasi-harmonic function $u$ (i.e., a function $u: (M,g)\longrightarrow \r$ with 
$\Delta u={\rm constant}\ne 0$) is an $f$-biharmonic function if and only if $f:(M,g)\longrightarrow \r$ is a harmonic function.
\end{proposition}
\begin{proof}
A straightforward computation gives the tension and the bitension fields of $u:(M,g)\longrightarrow \r$ as
\begin{equation}
\tau(u)=(\Delta u)\frac{\partial}{\partial t},\; {\rm and}\;\; \tau_2(u)=(\Delta^2 u)\frac{\partial}{\partial t}.
\end{equation}
Substituting these into $f$-biharmonic map equation (\ref{FBH}) and performing a further computation we obtain the $f$-biharmonic function equation (\ref{FBHF}). The last statement follows from Equation (\ref{FBHF}).
\end{proof}
\begin{corollary}\label{C25}
Any $f$-biharmonic function on a compact manifold $(M, g)$ is a constant function.
\end{corollary}
\begin{proof}
By Proposition \ref{P3}, $u$ is an $f$-biharmonic function if and only if $f\Delta u$ is a harmonic function. By the well known fact that any harmonic function on a compact manifold is constant we have $f\Delta u=C$, and hence 
\begin{equation}\label{CPT}
\Delta u=\frac{C}{f}
\end{equation}
since $f>0$ by our assumption. If $C=0$, then we have $\Delta u=0$ and hence $u$ is a harmonic function, so $u$ is a constant function in this case. If $C\ne 0$, then (\ref{CPT}) implies that $u$ is either a subharmonic or a superharmonic function since $f$ has a fixed sign with $f>0$. Again, the well-known fact that a subharmonic or superharmonic function on a compact manifold is constant implies that $u$ is constant. This completes the proof of the corollary.
\end{proof}
\begin{example}
Let $f:\r^3\setminus\{0\}\longrightarrow \r$ be the function $f(x, y, z)=1/\sqrt{x^2+y^2+z^2}$ and let $u:\r^3\setminus\{0\}\longrightarrow \r$ be the function given by $u(x,y,z)=x^2+y^2+z^2$. It is easily checked that $\Delta f=0,\;\;\Delta u=6$ and $\Delta^2 u=0$ and hence $f$ and $u$ satisfy Equation (\ref{FBHF}). So, $u(x,y,z)=x^2+y^2+z^2$ is an $f$-biharmonic function on $\r^3\setminus\{0\}$ for $f(x, y, z)=1/\sqrt{x^2+y^2+z^2}$. Clearly, this $f$-biharmonic function $u$ is not a harmonic function.\\
\end{example}
\begin{example}
Let $f, u :\r^3\setminus\{0\}\longrightarrow \r$ be the functions defined by $f(x, y, z)=\sqrt{x^2+y^2+z^2}$ and let $u(x, y, z)=\frac{x}{x^2+y^2+z^2}$. Then, we can check (see also Proposition \ref{INVERSE}) that $u$ is a proper $f$-biharmonic function which is neither harmonic nor biharmonic. \\
\end{example}
\begin{corollary}\label{RR}
Let $f, u:\r\longrightarrow \r$ be two functions with $f(x)>0, \;\forall\, x\in \r$. Then, $u$ is an $f$-biharmonic function if and only if
\begin{equation}\label{Ex2}
u(x)=\int\int\frac{Ax+B}{f}\,dxdx+Cx+D,
\end{equation}
where $A, B, C, D$ are arbitrary constants. In particular,
\begin{itemize}
\item[(1)] For $f(x)=1+x^2$, a function $u:\r\longrightarrow \r$ is $f$-biharmonic if and only if $u(x)=\frac{1}{2}(Ax-B)\ln (1+x^2)+(Bx+A)\arctan x +(C-A)x+D$, where $A, B, C, D$ are constants; and
\item[(2)] For, $f(x)=e^{-x}$, a function $u:\r\longrightarrow \r$ is $f$-biharmonic if and only if $u(x)=(Ax-2A+B)e^x+Cx +D$, where $A, B, C, D$ are constants.
\end{itemize}
\end{corollary}
\begin{proof}
In this case, the $f$-biharmonic equation (\ref{FBHF2}) reduces to $(f\,u'')''=0$ which has solution (\ref{Ex2}). Finally (1) and (2) are obtained by elementary integrations.
\end{proof}
\begin{remark}
It is easily checked that for $A\ne 0, B\ne 0$ the function $u(x)=(Ax-2A+B)e^x)+Cx +D$ is neither a harmonic nor a biharmonic function, so it provides many examples of proper $f$-biharmonic functions.
\end{remark}
\begin{theorem}\label{TE}
Any $f$-biharmonic map $\phi: (M^m, g)\longrightarrow \r^n$ from a compact manifold into a Euclidean space is a constant map.
\end{theorem}
\begin{proof}
Since the target manifold is a Euclidean space, the curvature is zero. If we write $\phi: (M^m, g)\longrightarrow \r^n$ as $\phi(p)=(\phi^1(p), \phi^2(p), \cdots, \phi^n(p))$, then we can easily check that 
\begin{eqnarray}
\tau(\phi)&=&(\Delta \phi^1, \; \Delta \phi^2, \;\cdots, \;\Delta \phi^n),\\\notag
\tau_2(\phi)&=&(\Delta^2 \phi^1,\; \Delta^2 \phi^2, \; \cdots, \;\Delta^2 \phi^n),{\rm and}\\\notag
\nabla^{\phi}_{{\rm grad} f}\tau(\phi)&=&(\nabla^{\phi}_{{\rm grad} f} \Delta \phi^1, \; \nabla^{\phi}_{{\rm grad} f} \Delta \phi^2, \;\cdots, \;\nabla^{\phi}_{{\rm grad} f} \Delta \phi^n).
\end{eqnarray}
It follows that the $f$-biharmonic map equation for $\phi$ becomes
\begin{eqnarray}
f\Delta^2 \phi^{\alpha}+(\Delta f) \Delta \phi^{\alpha}+2g({\rm grad} f, {\rm grad}\Delta \phi^{\alpha})=0,\;\; \alpha=1, 2, \cdots, n.
\end{eqnarray}
In other words, a map $\phi: (M^m, g)\longrightarrow \r^n$ from a manifold into a Euclidean space is and $f$-biharmonic map if and only if each of its component function is an $f$-biharmonic function. From this and Corollary \ref{C25} which states that any $f$-biharmonic function on a compact manifold is constant we obtain the theorem.
\end{proof}
\begin{proposition}\label{INVERSE}
The map $\phi: \r^m\setminus \{0\}\longrightarrow \r^m\setminus \{0\}$ with
$\phi(x)=\frac{x\;}{|x|^p}$ is an $f$-biharmonic map for $f(x)=|x|^k$ if and only if (i) $p=0$, or (ii) $p=m$, or (iii) $k=p+2$, or (iv) $k=p+2-m$.
In particular, for $m\ge 3$, the inversion $\phi: \r^m\setminus \{0\}\longrightarrow \r^m\setminus \{0\}$ with
$\phi(x)=\frac{x\;}{|x|^2}$ is a proper $f$-biharmonic map for $f(x)=|x|^4$. When $m\ne 4$, this inversion is also a proper $f$-biharmonic map for $f(x)=|X|^{4-m}$.
\end{proposition}
\begin{proof}
As we have seen in the proof of Theorem \ref{TE} that a map into a Euclidean space is an $f$-biharmonic map if and only if each of its component function is an $f$-biharmonic function. So, $\phi: \r^m\setminus \{0\}\longrightarrow \r^m\setminus \{0\}$ with
$\phi(x)=\frac{x\;}{|x|^p}$ is $f$-biharmonic if and only if the function
$u: \r^m\setminus \{0\}\longrightarrow \r$ with
$u(x)=x^i|x|^{-p}$ is an $f$-biharmonic function for any $i=1, 2, \cdots, m$. This, by Proposition \ref{P3}, is equivalent to the product $f\Delta u$ being a harmonic function. Using the formula $\triangle^{\mathbb{R}^{m}}(\left|x\right|^{\alpha})=\alpha(\alpha-2+m)
\left|x\right|^{\alpha-2}$ and a straightforward computation we have
\begin{eqnarray}
\triangle^{\mathbb{R}^{m}} u &=&\triangle^{\mathbb{R}^{m}} (x^i|x|^{-p}) = x^i\triangle^{\mathbb{R}^{m}} |x|^{-p}+2\langle {\rm grad}\, x^i, {\rm grad} |x|^{-p}\rangle
\\\notag
&=& p(p-m)x^i|x|^{-p-2}.
\end{eqnarray}
For $f(x)=|x|^k$, we have 
\begin{eqnarray}
\triangle^{\mathbb{R}^{m}} (f\triangle^{\mathbb{R}^{m}} u) &=& p(p-m)\triangle^{\mathbb{R}^{m}} (x^i|x|^{k-p-2}) \\\notag 
&=& p(p-m)[x^i\triangle^{\mathbb{R}^{m}} |x|^{k-p-2}+2\langle {\rm grad}\, x^i, {\rm grad} |x|^{k-p-2}\rangle]
\\\notag
&=&p (p-m)(k-p-2)(k-p+m-2)x^i|x|^{k-p-4}.
\end{eqnarray}
It follows that $u(x)=x^i|x|^{-p}$ is an $f$-biharmonic function with $f=|x|^k$ if and only if $p (p-m)(k-p-2)(k-p+m-2)=0$. Solving this equation we have (i) $p=0$, or (ii) $p=m$, or (iii) $k=p+2$, or (iv) $k=p+2-m$, from which the proposition follows.
\end{proof}
\begin{remark}
(A) One can check (see also \cite{BMO2}) that for the cases (i) $p=0$, and (ii) $p=m$, the maps $\phi=\frac{x}{|x|^p}$ are actually harmonic maps. We know that in these cases these maps are $f$-biharmonic for any $f$. For $k=0$ we have $f(x)=1$ and hence $f$-biharmonicity reduces to biharmonicity. In this case, (iii) and (iv) imply that $\phi=\frac{x}{|x|^p}$ is a proper biharmonic map if and only if $p=-2$, or $p=m-2$. Note that the case $p=-2$ was missed in the list of Remark 5.8 in \cite{BMO2}.\\
(B) For $ p\ne 0, m, {\rm and }\; k\ne 0$, the maps in cases (iii) and (iv) provide infinitely many example of proper $f$-biharmonic maps (i.e., which are neither harmonic nor biharmonic maps).\\
(C) It is well known that the inversion $\phi: \r^m\setminus \{0\}\longrightarrow \r^m\setminus \{0\}$, 
$\phi(x)=\frac{x\;}{|x|^2}$ is a conformal map between the same dimensional Euclidean spaces. Note that the $f$-biharmonic map equation for conformal maps between the same dimensional spaces was derived in \cite{Lu}, however, not a single example of such maps was found. Our Proposition \ref{INVERSE} shows that there are infinitely many proper $f$-biharmonic conformal diffeomorphisms and all but one of which are proper $f$-biharmonic for at least two different choices of $f$ functions. For some study on biharmonic diffeomorphisms see \cite{BFO1}.
\end{remark}
\section{$f$-Biharmonic submanifolds}
\begin{definition}
A submanifold in a Riemannian manifold is called an {\bf $f$-biharmonic submanifold} if the isometric immersion defining the submanifold is an $f$-biharmonic map.
\end{definition}
From the definition and the relationships among harmonic, biharmonic and $f$-biharmonic maps we have the following inclusion relationships
$$\{{\rm \bf Minimal\; submanifolds}\}\subset \{{\rm \bf Biharmonic \;submanifolds} \}\subset$$$$\subset \{{\rm \bf f-Biharmonic \;submanifolds}\}.$$
From now on we will call an $f$-biharmonic submanifold {\em a proper $f$-biharmonic submanifold} if it is neither a minimal nor a biharmonic submanifold.\\
\begin{theorem}\label{HSF}
Let $\phi:M^{m}\longrightarrow N^{m+1}$ be an isometric immersion
of codimension-one with mean curvature vector $\eta=H\xi$. Then
$\varphi$ is an $f$-biharmonic if and only if:
\begin{equation}\label{BHEq}
\begin{cases}
\Delta H-H |A|^{2}+H{\rm
Ric}^N(\xi,\xi)+H(\Delta f)/f+2({\rm grad} \ln f) H=0,\\
2A\,({\rm grad}\,H) +\frac{m}{2} {\rm grad}\, H^2
-2\, H \,({\rm Ric}^N\,(\xi))^{\top}+2H A\,({\rm grad}\,\ln f)=0,
\end{cases}
\end{equation}
where ${\rm Ric}^N : T_qN\longrightarrow T_qN$ denotes the Ricci
operator of the ambient space defined by $\langle {\rm Ric}^N\, (Z),
W\rangle={\rm Ric}^N (Z, W)$, $A$ is the shape operator of the
hypersurface with respect to the unit normal vector $\xi$, and $\Delta, \rm grad$ are the Laplace and the gradient operator of the hypersurface respectively.
\end{theorem}
\begin{proof}
It is well known that the tension field of the hypersurface is given by
\begin{equation}\label{TI}
\tau(\phi)=m H\xi.
\end{equation}
From Theorem 2.1 in \cite{Ou3} we have the bitension field of the hypersurface given by
\begin{eqnarray}\label{BTI}
\tau_{2}(\phi) &=&m\left(\Delta H-H|A|^{2}+H {\rm Ric}^N(\xi,\xi)\right)\xi\\\notag
&&-m\left( 2A\,({\rm grad}\,H) +\frac{m}{2} ({\rm
grad}\, H^2) -2\, H\,({\rm Ric}(\xi))^{\top}\right),
\end{eqnarray}
To compute the term $\nabla^{\phi}_{{\rm grad} f}\tau(\phi)$, we
choose a local orthonormal frame $\{{e_{i}}\}_{i=1,\ldots,m}$ on $M$
so that $\{{\rm d}\phi(e_{1}), \ldots, {\rm d}\phi(e_{m}),
\xi\}$ forms an adapted orthonormal frame of the ambient space defined
on the hypersurface. Identifying ${\rm
d}\phi(X)=X,\;\;\nabla^{\phi}_{X} W =\nabla^{N}_{X}W$ we have
\begin{eqnarray}\label{GI}
\nabla^{\phi}_{{\rm grad} f}\tau(\phi)=m \nabla^{N}_{{\rm grad} f} H\xi=m\{[({\rm grad} f)H]\xi-A({\rm grad} f )\}
\end{eqnarray}
Substituting (\ref{TI}), (\ref{BTI}) and (\ref{GI}) into the $f$-biharmonic map equation (\ref{FBH}) and simplifying the result we obtain the theorem.
\end{proof}
\begin{corollary}\label{CC}
A hypersurface $\phi: M^{m}\longrightarrow N^{m+1}(C)$ in a space form of constant sectional curvature $C$ is $f$-biharmonic if and only if its mean curvature function $H$ satisfies the following equation
\begin{equation}\label{Sm}
\begin{cases}
\Delta H-H |A|^{2}+mCH+H(\Delta f)/f+2({\rm grad} \ln f) H=0,\\
2A\,({\rm grad}\,H) +\frac{m}{2} {\rm grad}\, H^2
+2H A\,({\rm grad}\,\ln f)=0.
\end{cases}
\end{equation}
\end{corollary}
Similarly, we have
\begin{corollary}\label{CSMFLD}
A submanifold $\phi: M^{m}\longrightarrow N^{n}(C)$ in a space form of constant sectional curvature $C$ is $f$-biharmonic if and only if its mean curvature vector $H$ satisfies the following equation
\begin{equation}\label{CONT}
\begin{cases}
\Delta^{\bot}H-\frac{\Delta f}{f}H-2\nabla^{\perp}_{\rm grad \ln f}H+{\rm trace}\, B(-,A_{H}-)+Cm H=0\\
2 \,{\rm trace}\, A_{{\nabla_{(-)}^{\bot}}H}(-)+\frac{m}{2}\,{\rm
grad}\, (|H|^2)+2A_H(\rm grad\, ln f)=0,
\end{cases}
\end{equation}
\end{corollary}
where $\Delta^{\bot} H=-\rm Trace (\nabla^{\bot})^2 H$.
\begin{corollary}
A compact nonzero constant mean curvature $f$-biharmonic hypersurface $\phi: M^{m}\longrightarrow S^{m+1}$ in a sphere with $|A|^2=\rm constant$ is biharmonic.
\end{corollary}
\begin{proof}
Substituting $H=\rm constant\ne 0$ into the $f$-biharmonic hypersurface equation (\ref{Sm}) we have
\begin{equation}\label{Sm}
\begin{cases}
\Delta f=(|A|^2-m)f,\\
A\,({\rm grad}\,\ln f)=0.
\end{cases}
\end{equation}
If $|A|^2$ is constant, then we have either $|A|^2-m=0$ and in this case $f$ is a harmonic function, or $|A|^2-m\ne 0$. In the latter case we conclude that $f$ is either subharmonic or superharmonic since $f>0$. Since $M$ is compact, the well-known fact that any harmonic (subharmonic or superharmonic) function on a compact manifold is constant implies that $f$ is a constant function. Thus, the $f$-biharmonic hypersurface is actually biharmonic. 
\end{proof}

For classification of biharmonic submanifolds with parallel mean curvature vector and $|A|^2=\rm constant$ in sphere see \cite{BMO3}.\\

It is well known (see \cite{CI} and also \cite{Ji2}) that any biharmonic surface in Euclidean space $\r^3$ is minimal. In other words, there exists no proper biharmonic surfaces in $\r^3$. Our first question to ask is: does there exist a proper $f$-biharmonic surface in $\r^3$ ? We will show that there are infinitely many such surfaces in $\r^3$. We achieve this by using the interesting link between $f$-biharmonic surfaces and the biharmonic conformal immersions of surfaces in a $3$-manifold. For the study of biharmonic conformal immersions of surfaces in 3-manifolds we refer the readers to \cite{Ou2} and \cite{Ou5}. In particular, we recall that a surface (i.e., an isometric immersion) $\phi: M^2 \longrightarrow (N^3, h)$ is said to admit a biharmonic conformal immersion into $3$-manifold $(N^3, h)$ if there exists a function $\lambda: M^2\longrightarrow (0, \infty)$ such that the conformal immersion $\phi: (M^2, \lambda^{-2}\phi^*h) \longrightarrow (N^3, h)$ is biharmonic map. In this case, we also say that the surface $\phi: M^2 \longrightarrow (N^3, h)$ can be biharmonically conformally immersed into the $3$-manifold $(N^3, h)$ with conformal factor $\lambda$.
\begin{corollary}\label{CMC}
(1) A surface $\phi: M^2 \longrightarrow (N^3, h)$ in a $3$-manifold is an $f$-biharmonic surface if and only if the conformal immersion $\phi: (M^2, f^{-1}\phi^*h) \longrightarrow (N^3, h)$ is a biharmonic map, i.e., the surface can be biharmonically conformally immersed into $(N^3, h)$ with conformal factor $\lambda=f^{1/2}$. \\
(2) The circular cylinder $\phi:D=\{ (\theta,z)\in (0,2\pi)\times \r \}\longrightarrow (\r^3, \delta_0)$ with $\phi(\theta,z)=(R\cos\,\theta, R\sin\,\theta, z)$ is an $f$-biharmonic surface for any function $f$ from the family $f=\big(C_2e^{\pm z/R}-C_1C_2^{-1}R^2e^{\mp z/R}\big)/2$, where $C_1, C_2$ are constants.
\end{corollary}
\begin{proof}
Statement (1) follows from the definition of an $f$-biharmonic surface and Theorem \ref{CONF} whilst Statement (2) is obtained by using Statement (1) and Proposition 2 in \cite{Ou2}.
\end{proof}
\section{$f$-Biharmonic curves}
Another special case of $f$-biharmonic maps is an $f$-biharmonic curve.
\begin{lemma}
An arclength parametrized curve $\gamma: (a, b)\longrightarrow (N^m, g)$ is an $f$-biharmonic curve with a function $f: (a, b)\longrightarrow (0,\infty) $ if and only if
\begin{equation}\label{CURVE}
f(\nabla^N_{\gamma'}\nabla^N_{\gamma'}\nabla^N_{\gamma'}\gamma'-R^N(\gamma', \nabla^N_{\gamma'}\gamma')\gamma')+2f' \nabla^N_{\gamma'}\nabla^N_{\gamma'}\gamma'+f''\nabla^N_{\gamma'}\gamma'=0
\end{equation}
\end{lemma}
\begin{proof}
Let $\gamma=\gamma(s)$ be parametrized by arclength. Then, $e_1=\frac{\partial }{\partial s}$ is an orthonormal frame on $((a, b), ds^2)$ and $d\gamma (e_1)=d\gamma (\frac{\partial }{\partial s})=\gamma'$. Thus, the tension field of the curve is given by $\tau(\gamma)=\nabla^{\gamma}_{e_1}d\gamma(e_1)=\nabla^N_{\gamma'}\gamma'$. It is also easy to see that for a function $f: (a, b)\longrightarrow (0,\infty) $, $\Delta f=f''$ and $\nabla^{\gamma}_{{\rm grad}f}\tau(\gamma)=f'\nabla^N_{\gamma'}\nabla^N_{\gamma'}\gamma'$. Substituting these into the $f$-biharmonic map equation gives the lemma.
\end{proof}
\begin{theorem}\label{CN}
An arclength parametrized curve $\gamma: (a, b)\longrightarrow N^n(C)$ in an $n$-dimensional space form is a proper $f$-biharmonic curve if and only if one of the following cases happens:
\begin{itemize}
\item[(i)] $\kappa_2= 0$, $f=c_1\kappa_1^{-3/2}$ and the curvature $\kappa_1$ solves the following ODE
\begin{equation}
3\kappa_1'^2-2\kappa_1\kappa_1''=4\kappa_1^2(\kappa_1^2-C).
\end{equation}
\item[(ii)] $\kappa_2\ne 0,\;\; \kappa_3=0,\;\;\kappa_2/\kappa_1=c_3$, $f=c_1\kappa_1^{-3/2}$, and the curvature $\kappa_1$ solves the following ODE
\begin{equation}
3\kappa_1'^2-2\kappa_1\kappa_1''=4\kappa_1^2[(1+c_3^2)\kappa_1^2-C].
\end{equation}
\end{itemize}
\end{theorem}
\begin{proof}
Let $\gamma: (a, b)\longrightarrow N^n(C)$ be a curve with arclength parametrization. Let $\{F_i, \; i=1, 2, \cdots,n\}$ be the Frenet frame along the curve $\gamma(s)$, which is obtained as the orthonormalisation of the $n$-tuple
$\{\nabla^{(k)}_{\frac{\partial}{\partial s}} d\gamma ( \frac{\partial}{\partial s} )|\; k=1, 2, \cdots, n\}$. Then we have the following Frenet formula (see, e.g., \cite{La}) along the curve:
\begin{equation}
\begin{cases}
\nabla^{\gamma}_{\frac{\partial}{\partial s}}F_1=\kappa_1 F_2,\\
\nabla^{\gamma}_{\frac{\partial}{\partial s}}F_i=-\kappa_{i-1} F_{i-1}+\kappa_i F_{i+1},\; \forall \;i=2, 3, \cdots, n-1,\\
\nabla^{\gamma}_{\frac{\partial}{\partial s}}F_i=-\kappa_{n-1}F_{n-1},
\end{cases}
\end{equation}
where $\{\kappa _1, \kappa_2, \cdots, \kappa_{n-1}\}$ are the curvatures of the curve $\gamma$.\\
Using this formula and a straightforward computation one finds the tension and the bitension fields of the curve given by
\begin{eqnarray}\notag
\tau(\gamma)&=&\nabla^N_{\gamma'}\gamma'=\kappa_1 F_2,\\\notag
\nabla^N_{\gamma'}\nabla^N_{\gamma'}\gamma'&=&-\kappa_1^2F_1+\kappa_1' F_2 +\kappa_1\kappa_2 F_3, \; {\rm and}\\\notag
\tau_2(\gamma) &=&-3\kappa_1\kappa_1' F_1+(\kappa_1''-\kappa_1\kappa_2^2-\kappa_1^3+\kappa_1 C)F_2+\\\notag
&&+(2\kappa_1'\kappa_2+\kappa_1\kappa_2')F_3+\kappa_1\kappa_2\kappa_3F_4.
\end{eqnarray}
Substituting these into the $f$-biharmonic curve equation (\ref{CURVE}) and comparing the coefficients of both sides we have
\begin{equation}\label{37}
\begin{cases}
-3\kappa_1\kappa_1'-2\kappa_1^2f'/f=0,\\
\kappa_1''-\kappa_1\kappa_2^2-\kappa_1^3+\kappa_1 C+\kappa_1f''/f+2\kappa_1'f'/f=0,\\
2\kappa_1'\kappa_2+\kappa_1\kappa_2'+2\kappa_1\kappa_2 f'/f=0,\\
\kappa_1\kappa_2\kappa_3=0.
\end{cases}
\end{equation}
It is easy to see that if $\kappa_1=\rm constant\ne 0$, then the first equation of (\ref{37}) implies that $f$ is constant and the curve $\gamma$ is biharmonic. Also, if $\kappa_2=\rm constant\ne 0$, then the first and the third equations (\ref{37}) imply that $f$ is constant and hence the curve $\gamma$ is biharmonic again.\\
Now, if $\kappa_2= 0$, then the $f$-biharmonic curve equation (\ref{37}) is equivalent to 
\begin{equation}\label{GD28}
\begin{cases}
3\frac{\kappa_1'}{\kappa_1}+2 \frac{f'}{f}=0,\\
\frac{\kappa_1''}{\kappa_1}-\kappa_1^2+ C+\frac{f''}{f}+2\frac{\kappa_1'}{\kappa_1}\frac{ f'}{f}=0.
\end{cases}
\end{equation}
Integrating the first equation of (\ref{GD28}) and substituting the result in to the second we obtain statements in Case (i).\\
Finally, if $\kappa_1\ne \rm constant$ and $\kappa_2\ne \rm constant$, then the system (\ref{37}) is equivalent to
\begin{equation}\label{C4}
\begin{cases}
f^2\kappa_1^3=c_1^2,\\
(f\kappa_1)''=f\kappa_1 (\kappa_2^2+\kappa_1^2-C),\\
f^2\kappa^2\kappa_2=c_2,\\
\kappa_3=0.
\end{cases}
\end{equation}
Solving the first equation of (\ref{C4}) we obtain $f=c_1\kappa_1^{-3/2}$. Substituting the first equation into the third one in (\ref{C4}) we obtain $\kappa_2/\kappa_1=c_3$. Finally, substituting $\kappa_2/\kappa_1=c_3$ and $f\kappa_1=c_1\kappa_1^{-1/2}$ into the second equation of (\ref{C4}) we obtain the results stated in Case (ii). This completes the proof of the theorem.
\end{proof}
From the proof of Theorem \ref{CN} we have
\begin{corollary}
An arclength parametrized curve $\gamma: (a, b)\longrightarrow N^n(C)$ in an $n$-dimensional space form with constant geodesic curvature is biharmonic.
\end{corollary}
It is known (\cite{Di}) that any biharmonic curve in a Euclidean space is a geodesic. It would be interesting to know if there is any proper $f$-biharmonic curve in a Euclidean space. Our next theorem gives a complete classification of proper $f$-biharmonic curves in $\r^3$ which, together with the fundamental theorem for curves in $\r^3$, can be used to produce many examples of proper $f$-biharmonic curves in a $3$-dimensional Euclidean space.
\begin{theorem}\label{R3}
An arclength parametrized curve $\gamma: (a, b)\longrightarrow \r^3$ in a $3$-dimensional Euclidean space is a proper $f$-biharmonic curve if and only if 
\begin{itemize}
\item[(i)] $\gamma$ is a planar curve with $\tau(s)= 0$, $\kappa(s)= \frac{4c_2}{16+(c_2 s+c_3)^2}$, and $f=c_1\kappa^{-3/2}$, where $c_1>0, c_2> 0, {\rm and}\; c_3$ are constants, or
\item[(ii)] $\gamma$ is a general helix with $\kappa(s)= \frac{4c_2}{16(1+c^2)+(c_2 s+c_3)^2},\;\;\tau/\kappa=c$, and $f=c_1\kappa^{-3/2}$, where $c\ne 0, c_1>0, c_2> 0, {\rm and}\; c_3$ are constants.
\end{itemize}
\end{theorem}
\begin{proof}
For the arclength parameterized curve $\gamma: (a, b)\longrightarrow \r^3$, we have the curvature $\kappa=\kappa_1$ and the torsion $\tau=\kappa_2$. Applying Theorem \ref{CN} with $C=0$ we conclude that the curve $\gamma$ is a proper $f$-biharmonic curve if and only if
\begin{itemize}
\item[(i)] $\tau= 0$, $f=c_1\kappa^{-3/2}$ and the curvature $\kappa$ solves the following ODE
\begin{equation}
3\kappa'^2-2\kappa\kappa''=4\kappa^4, \; {\rm or}
\end{equation}
\item[(ii)] $\tau\ne 0,\;\;\tau/\kappa=c$, $f=c_1\kappa_1^{-3/2}$, and the curvature $\kappa$ solves the following ODE
\begin{equation}
3\kappa'^2-2\kappa\kappa''=4(1+c^2)\kappa^4.
\end{equation}
\end{itemize}
Solving the ODEs in each case and noting that $\tau=0$ means the curve is planar and $\tau/\kappa=c$ means the curve is a general helix ( Lancret Theorem, see, e.g., \cite{Ba}) we obtain the theorem.
\end{proof}
\begin{remark}
(A) Recall that the fundamental theorem for curves in $\r^3$ states that for any given functions $p, q: [s_0,s_1]\longrightarrow \r$ with $p(s)>0,\; \forall \,s \in [s_0,s_1]$, there exists a unique (up to a rigid motion) curve in $\r^3$ whose curvature and torsion take on the prescribed functions $\kappa(s)=p(s), \; \tau(s)=q(s)$. This, together with our Theorem \ref{R3}, implies that there are many examples of proper $f$-biharmonic curves in $\r^3$.\\
(B) Our classification theorem also implies that proper $f$-biharmonic curves in $\r^3$ must be special subclasses of planar curves or general helices in $\r^3$. As the following example shows that there are  general helices which are not proper $f$-biharmonic curves.
\end{remark}
\begin{example}
The general helix $\gamma:I\longrightarrow \r^3$ with $\gamma (s)=(\frac{2}{3} (1+s/2)^{3/2}, \frac{2}{3} (1-s/2)^{3/2}, s/\sqrt{2} )$ is never an $f$-biharmonic curve for any function $f$.\\

 In fact, one can easily check that $|\gamma'(s)|=1$ so $s$ is an arclength parameter for the curve. A straightforward computation gives $\kappa(s)=\tau(s)=\frac{1}{2\sqrt{2}\sqrt{4-s^2}}$. So, the curve is indeed a general helix with $\tau/\kappa=1$. Since the curvature is not of the form given in Case (ii) of Theorem \ref{R3} we conclude that the helix is never an $f$-biharmonic curve for any $f$.\\
\end{example}
Finally, we give an example of a proper $f$-biharmonic planar curve to close this section.
\begin{example}
The planar curve $\gamma(s)=(4\ln |\sqrt{16+s^2}+s|, \sqrt{16+s^2})$ is a proper $f$-biharmonic curve. \\

In fact, we can check that 
\begin{eqnarray}
\gamma'(s)=\left(\frac{4}{\sqrt{16+s^2},}\; \frac{s}{\sqrt{16+s^2},}\right),\; {\rm and}\;\; |\gamma'(s)|=1.
\end{eqnarray}
So $s$ is the arclength parameter of the curve. In this case, we have the curvature $\kappa(s)=|\gamma''(s)|=\frac{4}{16+s^2}$ and, by (i) of Theorem \ref{R3}, the curve $\gamma$ is a proper $f$-biharmonic curve with $f=8c_1(16+s^2)^{3/2}$ for some constant $c_1>0$.
\end{example}

\end{document}